\documentclass[12pt]{article}

\usepackage[T1]{fontenc}
\usepackage{lmodern}
\usepackage{amsmath,amssymb,amsfonts,amsthm,mathtools}
\usepackage{enumitem}
\usepackage{needspace}
\usepackage[pdfauthor={Chunqiu Fang and Rongxing Xu},
    pdftitle={On the maximum number of triangles in tripartite graphs with no 4-cycles between any two parts},
  pdfkeywords={tripartite graphs, 4-cycles, planar polynomials, finite fields},
  pdfstartview=XYZ,bookmarks=true,colorlinks=true,
  linkcolor=blue,urlcolor=blue,citecolor=blue,
  linktocpage=true,hyperindex=true]{hyperref}

\allowdisplaybreaks
\numberwithin{equation}{section}
\setlist[enumerate]{leftmargin=2.35em,itemsep=0.2em,topsep=0.3em}
\setlist[itemize]{leftmargin=2em,itemsep=0.2em,topsep=0.3em}

\newtheorem{theorem}{Theorem}[section]
\newtheorem{lemma}[theorem]{Lemma}

\newtheorem{corollary}[theorem]{Corollary}

\newtheorem{observation}[theorem]{Observation}
\newtheorem{problem}[theorem]{Problem}
\theoremstyle{definition}
\newtheorem{definition}[theorem]{Definition}

\newcommand{\F}{\mathbb F}
\newcommand{\sourcecite}[2]{\cite{#1}, #2}

\title{\large{\bfseries On the maximum number of triangles in tripartite
graphs with no $4$-cycles between any two parts}}
\author{
  Chunqiu Fang \thanks{School of Computer Science and Technology, Dongguan University of Technology, Dongguan, Guangdong, 523808, China. Email: \texttt{fcq15@tsinghua.org.cn}. Supported by the National Natural Science Foundation for Young Scientists of China (Grant No.~12301435).}
  \quad
  Rongxing Xu\thanks{School of Mathematical Sciences, Zhejiang Normal University, Jinhua, Zhejiang, 321000, China. Email: \texttt{xurongxing@zjnu.edu.cn}. Supported by the National Natural Science Foundation for Young Scientists of China (Grant No.~12401472) and the Zhejiang Provincial Natural Science Foundation of China (Grant No.~LQN25A010011).}
}
\date{}

\begin{document}

\maketitle

\begin{abstract}
Let $G$ be a $3$-partite graph with $k$ vertices in each part such that the bipartite graph induced by any two parts contains no cycle of length four. Fischer and Matou\v{s}ek [J. Combin. Theory Ser. A, 2001] asked for the maximum number of triangles in such a graph. They obtained the lower bound $(1-o(1))k^{3/2}$ and the upper bound $k^{7/4}+O(k^{3/2})$. Coulter, Matthews and Timmons [J. Combin. Theory Ser. B, 2018] later constructed such graphs using planar polynomials over finite fields and improved the lower bound to $(1-o(1))k^{5/3}$. In this note, we use a new triple of planar polynomials and further improve the lower bound to $(1-o(1))k^{17/10}$.
\end{abstract}


\section{Introduction}

Fischer and Matou\v{s}ek~\cite{FischerMatousek2001} posed the following graph problem while studying the extremal size of families of multivalued functions with bounded Natarajan dimension.

\begin{problem}\label{prob:FM}
Let $G$ be a $3$-partite graph with $k$ vertices in each part. Suppose that the bipartite graph induced by any two parts contains no cycle of length four. Determine how many triangles can appear in such a graph.
\end{problem}

Let $\triangle(k)$ be the maximum number of triangles in a graph satisfying Problem~\ref{prob:FM}. A construction of K\H{o}v\'ari, S\'os and Tur\'an~\cite{KovariSosTuran1954} yields $C_4$-free bipartite graphs with $k$ vertices in each part and $(1-o(1))k^{3/2}$ edges. Fischer and Matou\v{s}ek placed one such graph between two parts. They extended every edge to a triangle through a single vertex in the third part. This construction shows that $\triangle(k)\ge (1-o(1))k^{3/2}$. They also proved the upper bound $\triangle(k)\le k^{7/4}+O(k^{3/2})$~\cite{FischerMatousek2001}.
Coulter, Matthews and Timmons~\cite{CMT2018} later constructed such graphs using planar polynomials over finite fields and improved the lower bound to $\triangle(k)\ge(1-o(1))k^{5/3}$.

In this note, we use a new triple of planar polynomials to prove the following stronger lower bound.

\begin{theorem}\label{thm:main}
If $q$ is a power of an odd prime, then
$\triangle(q^{10})\ge q^{17}+q^{15}$.
\end{theorem}

\begin{corollary}\label{cor:asymptotic}
As $k\to\infty$,
\[
        \triangle(k)\ge (1-o(1))k^{17/10}.
\]
\end{corollary}

\begin{proof}
Baker, Harman and Pintz~\cite{BakerHarmanPintz2001} proved that the interval $[x-x^{0.525},x]$ contains a prime whenever $x$ is sufficiently large. For each sufficiently large $k$, apply this result with $x=k^{1/10}$ and choose a prime $p=p(k)$ satisfying $k^{1/10}-k^{0.0525}\le p\le k^{1/10}$. The lower bound on $p$ tends to infinity. Thus $p$ is odd for all sufficiently large $k$. By Theorem~\ref{thm:main} with $q=p$, we have a graph with $p^{10}$ vertices in each part and at least $p^{17}+p^{15}$ triangles. We add isolated vertices until each part has $k$ vertices. The bounds on $p$ imply that $p/k^{1/10}\to1$. Hence $p^{17}+p^{15}=(1+o(1))k^{17/10}$, and the desired lower bound follows.
\end{proof}

In Section~\ref{sec:preliminaries}, we review the finite-field facts needed for the construction and establish a root-counting consequence of a five-periodic recurrence. In Section~\ref{sec:construction}, we define the graph and prove Theorem~\ref{thm:main}.

Throughout the rest of the paper, $q$ denotes a power of an odd prime. For standard finite-field terminology and results not recalled here, we refer the reader to \cite{LidlNiederreiter1997}.

\section{Preliminaries}\label{sec:preliminaries}

In this section, we first recall some standard facts about finite fields and polynomial maps. We then introduce planar polynomials and use them to construct $C_4$-free bipartite graphs. Finally, we construct the auxiliary $q$-polynomial used in Section~\ref{sec:construction}.

\subsection{Finite fields and polynomial maps}

We use the following standard finite-field facts.

\begin{lemma}\label{lem:finite-field-facts}
The following statements hold.
\begin{enumerate}[label=\textup{(\arabic*)},ref=(\arabic*)]
\item\label{part:finite-field-roots}
For every positive integer $n$, every extension field $F$ of $\F_{q^n}$ and every $x\in F$, we have $x\in\F_{q^n}$ if and only if $x^{q^n}=x$.~\textup{(A corollary of Lemma~2.4 of \cite{LidlNiederreiter1997}, p.~49)}
\item\label{part:frobenius-identities}
In any extension field $F$ of $\F_q$, for all $x,y\in F$, $c\in\F_q$ and integers $i\ge0$,
\[
 (x+y)^{q^i}=x^{q^i}+y^{q^i},
 \qquad
 (cx)^{q^i}=c x^{q^i}.
\]
\textup{(By Theorem~1.46 in \cite{LidlNiederreiter1997}, p.~16, for the first equality and repeated application of part~\ref{part:finite-field-roots} for the second)}
\end{enumerate}
\end{lemma}

%
For a positive integer $m$ and a nonnegative integer $s$, a polynomial of the form
\[
        L(X)=\sum_{i=0}^{s}a_iX^{q^i},
        \qquad a_0,\ldots,a_s\in\F_{q^m},
\]
is called a \emph{$q$-polynomial over $\F_{q^m}$}.

\Needspace{7\baselineskip}
By Lemma~\ref{lem:finite-field-facts}\ref{part:frobenius-identities} and the finiteness of $\F_{q^m}$, we have the following observation.

\begin{observation}\label{obs:q-polynomial-properties}
For every positive integer $m$ and every $q$-polynomial $L$ over $\F_{q^m}$,
\[
        L(x+y)=L(x)+L(y)
\]
for all $x,y\in\F_{q^m}$. Moreover, $L$ is bijective if and only if its only root in $\F_{q^m}$ is $0$.
\end{observation}

We also need the formal derivative of a polynomial. If $F$ is a field and
$f(X)=\sum_{i=0}^{d}c_i X^i\in F[X]$, its \emph{formal derivative} is
\[
        f'(X)=\sum_{i=1}^{d}i c_i X^{i-1}.
\]

The following lemma records two standard facts about polynomials.

\begin{lemma}\label{lem:polynomial-facts}
Let $F$ be a field.
\begin{enumerate}[label=\textup{(\arabic*)},ref=(\arabic*)]
\item\label{part:polynomial-factorization}
If $f\in F[X]$ has positive degree $d$, then there is an extension field $F_1$ of $F$ containing elements $\alpha_1,\ldots,\alpha_d$ such that
\[
        f(X)=c\prod_{j=1}^{d}(X-\alpha_j),
\]
where $c$ is the coefficient of $X^d$ in $f$.~\textup{(\sourcecite{LidlNiederreiter1997}{Theorem~1.91, p.~35})}
\item\label{part:repeated-root}
If $f\in F[X]$ is nonzero and $\alpha\in F$, then $\alpha$ is a repeated root of $f$ if and only if it is a root of both $f$ and $f'$.~\textup{(\sourcecite{LidlNiederreiter1997}{Theorem~1.68, p.~27})}
\end{enumerate}
\end{lemma}

\subsection{Planar polynomials}

Planar functions were introduced by Dembowski and Ostrom~\cite{DembowskiOstrom1968} in connection with the construction of affine planes. We use the following standard definition of a planar polynomial over a finite field~\cite{CMT2018}.

\begin{definition}\label{def:planar}
Let $F$ be a finite field and let $f\in F[X]$. The polynomial $f$ is \emph{planar over $F$} if, for every $a\in F^*=F\setminus\{0\}$, the map $x\mapsto f(x+a)-f(x)$ is a bijection from $F$ to itself.
\end{definition}

The next lemma records a basic stability property of planar polynomials.

\begin{lemma}\label{lem:q-polynomial-perturbation}
Let $m$ be a positive integer. If $f\in\F_{q^m}[X]$ is planar over $\F_{q^m}$, $c\in\F_{q^m}^*$ and $L$ is a $q$-polynomial over $\F_{q^m}$, then $cf+L$ is planar over $\F_{q^m}$.
\end{lemma}

\begin{proof}
By Observation~\ref{obs:q-polynomial-properties}, for every $a\in\F_{q^m}^*$ we have
\[
 (cf+L)(x+a)-(cf+L)(x)
 =c\bigl(f(x+a)-f(x)\bigr)+L(a).
\]
Since $x\mapsto f(x+a)-f(x)$ is bijective, its nonzero scalar multiple followed by translation by $L(a)$ is also bijective.
\end{proof}

\begin{lemma}\label{lem:planar}
The polynomial $P(X)=X^{q^2+1}$ is planar over $\F_{q^5}$. Furthermore,
\[
        2P(x)+2P(y)=P(x+y)+P(x-y)
\]
for all $x,y\in\F_{q^5}$.
\end{lemma}

\begin{proof}
For $a\in\F_{q^5}^*$, consider the $q$-polynomial
\[
        \lambda_a(X)=aX^{q^2}+a^{q^2}X.
\]
By Lemma~\ref{lem:finite-field-facts}\ref{part:frobenius-identities}, applied with $i=2$, for every $x\in\F_{q^5}$ we have
\begin{align*}
P(x+a)-P(x)
&=(x+a)^{q^2+1}-x^{q^2+1}\\
&=(x^{q^2}+a^{q^2})(x+a)-x^{q^2+1}\\
&=ax^{q^2}+a^{q^2}x+a^{q^2+1}
 =\lambda_a(x)+a^{q^2+1}.
\end{align*}

It remains to show that $\lambda_a$ is bijective.
Suppose now that $x\in\F_{q^5}$ satisfies $\lambda_a(x)=0$. Since $a\ne0$, write $x=at$ with $t\in\F_{q^5}$. Then
\[
        0=\lambda_a(at)=a^{q^2+1}(t^{q^2}+t),
\]
so $t^{q^2}=-t$. Since $q$ is odd, so is $q^2$. Hence $t^{q^4}=(-t)^{q^2}=-t^{q^2}=t$. By taking the $q$-th power of both sides of $t^{q^4}=t$, we have $t^{q^5}=t^q$. Since $t\in\F_{q^5}$, we also have $t^{q^5}=t$ by Lemma~\ref{lem:finite-field-facts}\ref{part:finite-field-roots}, applied with $n=5$. Therefore $t^q=t$. By taking the $q$-th power of both sides of $t^q=t$, we have $t^{q^2}=t^q=t$. Comparing this with $t^{q^2}=-t$, we obtain $2t=0$. The field $\F_{q^5}$ has odd characteristic, so $t=0$. Thus $0$ is the only root of $\lambda_a$ in $\F_{q^5}$. By Observation~\ref{obs:q-polynomial-properties}, $\lambda_a$ is bijective. The map
\[
        x\longmapsto P(x+a)-P(x)
                =\lambda_a(x)+a^{q^2+1}
\]
is a translation of $\lambda_a$, so it is also bijective. Hence $P$ is planar.

By Lemma~\ref{lem:finite-field-facts}\ref{part:frobenius-identities}, applied with $i=2$, for $x,y\in\F_{q^5}$ we have
\begin{align*}
P(x+y)&=(x^{q^2}+y^{q^2})(x+y)=P(x)+P(y)+x^{q^2}y+xy^{q^2},\\
P(x-y)&=(x^{q^2}-y^{q^2})(x-y)=P(x)+P(y)-x^{q^2}y-xy^{q^2}.
\end{align*}
The sum of these two identities is the second assertion.
\end{proof}

We next construct a $C_4$-free bipartite graph from a planar polynomial over a finite field. Similar finite-field constructions have been used to study several problems in extremal combinatorics \cite{AllenKeevashSudakovVerstraete2014,CMT2018,LvLuFang2020,LvLuFang2022,Timmons2017,TimmonsVerstraete2015}.

\begin{lemma}\label{lem:planar-graph}
Let $F$ be a finite field and let $f\in F[X]$ be planar over $F$. Let $A$ and $B$ be disjoint copies of $F^2$. For each $(x,y)\in F^2$, denote its copies in $A$ and $B$ by $(x,y)_A$ and $(x,y)_B$, respectively. Let $G$ be the bipartite graph with vertex set $A\cup B$ and edge set
\[
        E(G)=\bigl\{\{(x,y)_A,(x+t,y+f(t))_B\}:x,y,t\in F\bigr\}.
\]
Then $G$ does not contain a $4$-cycle.
\end{lemma}

\begin{proof}
Suppose to the contrary that two distinct vertices $u_1,u_2\in A$ have two distinct common neighbors $v_1,v_2\in B$. In the following coordinate calculations, we identify each vertex with its underlying element of $F^2$. For suitable $x_1,x_2,y_1,y_2\in F$,
\begin{align*}
v_1&=u_1+(x_1,f(x_1))=u_2+(y_1,f(y_1)),\\
v_2&=u_1+(x_2,f(x_2))=u_2+(y_2,f(y_2)).
\end{align*}
Thus we have
\[
 u_2-u_1
 =(x_1-y_1,f(x_1)-f(y_1))
 =(x_2-y_2,f(x_2)-f(y_2)).
\]
Comparing the first coordinates, we have
$x_1-y_1=x_2-y_2$. Let $t$ denote this common value.
Suppose that $t=0$. Then $x_1=y_1$ and $x_2=y_2$. In particular, the preceding vector equality becomes
\[
 u_2-u_1=(x_1-y_1,f(x_1)-f(y_1))=(0,0).
\]
It follows that $u_1=u_2$, contrary to the choice of two distinct vertices $u_1$ and $u_2$. Therefore, $t\ne0$.
Comparing the second coordinates, we obtain
\[
f(y_1+t)-f(y_1)=f(y_2+t)-f(y_2).
\]
By Definition~\ref{def:planar}, the map
$x\mapsto f(x+t)-f(x)$ is injective. Hence $y_1=y_2$. Since $x_1-y_1=x_2-y_2$, we also have $x_1=x_2$. It follows that $v_1=v_2$, a contradiction. Having two vertices in one part with two common neighbors is equivalent to having a $4$-cycle, so $G$ is $C_4$-free.
\end{proof}

\subsection{An auxiliary \texorpdfstring{$q$}{q}-polynomial}

Before constructing the auxiliary $q$-polynomial used in the proof of Theorem~\ref{thm:main}, we recall a five-periodic recurrence due to Lyness~\cite{Lyness1942}.

\begin{lemma}\label{lem:lyness}
Let $F$ be a field and let $(y_i)_{i\ge0}$ be a sequence of nonzero elements of $F$ satisfying $y_i y_{i+2}=1+y_{i+1}$ for every $i\ge0$. Then $y_{i+5}=y_i$ for every $i\ge0$.
\end{lemma}

\begin{proof}
Fix $i\ge0$, and let $a=y_i$ and $b=y_{i+1}$. By successive substitution in the recurrence, we have
\[
 y_{i+2}=\frac{b+1}{a},\qquad
 y_{i+3}=\frac{a+b+1}{ab},\qquad
 y_{i+4}=\frac{a+1}{b},\qquad
 y_{i+5}=a=y_i,\qquad
 y_{i+6}=b=y_{i+1}.
\]
All denominators and cancelled factors are nonzero because every term of the sequence is nonzero. Thus
$(y_{i+5},y_{i+6})=(y_i,y_{i+1})$. Since two consecutive terms determine every subsequent term through $y_{j+2}=(1+y_{j+1})/y_j$, the sequence repeats after five terms. In particular, $y_{i+5}=y_i$ for every $i\ge0$.
\end{proof}

We now use this five-periodicity to construct the auxiliary $q$-polynomial.

\begin{lemma}\label{lem:algebraic}
Let $H\in\F_q[X]$ be the polynomial $H(X)=X^{q^2+1}-X^q-1$.
Then $H$ has exactly $q^2+1$ distinct roots in $\F_{q^5}$. If $\rho$ is any root of $H$ and
\[
        L_\rho(X)=X^q-\rho X^{q^2}-\rho^{q^2}X,
\]
then $L_\rho$ is a $q$-polynomial over $\F_{q^5}$, and there are exactly $q^2+1$ elements $x\in\F_{q^5}$ satisfying $L_\rho(x)=x^{q^2+1}$.
\end{lemma}

\begin{proof}
By Lemma~\ref{lem:polynomial-facts}\ref{part:polynomial-factorization}, applied to $H$ as a polynomial over $\F_{q^5}$, there is an extension field $F$ of $\F_{q^5}$ containing elements $\alpha_1,\ldots,\alpha_{q^2+1}$ such that
\[
        H(X)=\prod_{j=1}^{q^2+1}(X-\alpha_j).
\]
Let $\alpha$ be any root of $H$ in $F$. Since $H(0)\ne0$, we have $\alpha\ne0$. As $H(\alpha)=0$, we have $\alpha^{q^2+1}=\alpha^q+1$.
For every $i\ge0$, taking the $q^i$-th power of both sides and applying Lemma~\ref{lem:finite-field-facts}\ref{part:frobenius-identities}, we obtain $\alpha^{q^i}\alpha^{q^{i+2}}=1+\alpha^{q^{i+1}}$. Therefore the sequence $y_i=\alpha^{q^i}$ consists of nonzero elements and satisfies the hypothesis of Lemma~\ref{lem:lyness}. Hence $y_5=y_0$, or equivalently $\alpha^{q^5}=\alpha$. Since $F$ is an extension field of $\F_{q^5}$, Lemma~\ref{lem:finite-field-facts}\ref{part:finite-field-roots}, applied with $n=5$, implies $\alpha\in\F_{q^5}$. The root $\alpha$ was arbitrary, so $\alpha_1,\ldots,\alpha_{q^2+1}$ all belong to $\F_{q^5}$.

Since the characteristic of $\F_{q^5}$ divides $q$,
\[
        H'(X)=(q^2+1)X^{q^2}-qX^{q-1}=X^{q^2}.
\]
The only root of $H'$ in $\F_{q^5}$ is $0$, whereas $H(0)=-1$. Thus $H$ and $H'$ have no common root in $\F_{q^5}$. Since all roots of $H$ lie in $\F_{q^5}$, Lemma~\ref{lem:polynomial-facts}\ref{part:repeated-root} implies that none of them is repeated. Hence $\alpha_1,\ldots,\alpha_{q^2+1}$ are distinct, and $H$ has exactly $q^2+1$ distinct roots in $\F_{q^5}$.

Let $\rho$ be any one of these roots. By definition, $L_\rho$ is a $q$-polynomial over $\F_{q^5}$. By Lemma~\ref{lem:finite-field-facts}\ref{part:frobenius-identities} and the equality $H(\rho)=0$, for every $x\in\F_{q^5}$ we have
\begin{align*}
 H(x+\rho)
 &=(x+\rho)^{q^2+1}-(x+\rho)^q-1\\
 &=x^{q^2+1}-x^q+\rho x^{q^2}+\rho^{q^2}x
   +\rho^{q^2+1}-\rho^q-1\\
 &=x^{q^2+1}-L_\rho(x).
\end{align*}
Translation by $\rho$ is a bijection on $\F_{q^5}$. Consequently, there are exactly $q^2+1$ elements $x\in\F_{q^5}$ satisfying $L_\rho(x)=x^{q^2+1}$.
\end{proof}

\section{Proof of Theorem~\ref{thm:main}}\label{sec:construction}

In this section, we prove Theorem~\ref{thm:main}. We construct a $3$-partite graph with $q^{10}$ vertices in each part, no $4$-cycle between any two parts, and $q^{17}+q^{15}$ triangles.

By Lemma~\ref{lem:algebraic}, there exists $\rho\in\F_{q^5}$ satisfying $\rho^{q^2+1}-\rho^q-1=0$. For this $\rho$, define
\[
        P(X)=X^{q^2+1},
        \qquad
        L_\rho(X)=X^q-\rho X^{q^2}-\rho^{q^2}X.
\]
By the same lemma, $L_\rho$ is a $q$-polynomial over $\F_{q^5}$. There are exactly $q^2+1$ elements $w\in\F_{q^5}$ satisfying $P(w)=L_\rho(w)$. With $P$ and $L_\rho$ fixed, let
\[
 f(X)=2P(X)-L_\rho(X),\qquad
 g(X)=2P(X)+L_\rho(X),\qquad
 h(X)=-P(X).
\]
\Needspace{9\baselineskip}
Let $A,B,C$ be disjoint copies of $\F_{q^5}^2$. We denote their vertices by $(x,y)_A$, $(x,y)_B$ and $(x,y)_C$, respectively. Let $G_q(\rho)$ be the graph on $A\cup B\cup C$, where for all $x,y,z\in\F_{q^5}$,
\begin{itemize}
\item $(x,y)_A$ is adjacent to $(x+z,y+f(z))_B$.
\item $(x,y)_B$ is adjacent to $(x+z,y+g(z))_C$.
\item $(x,y)_C$ is adjacent to $(x+z,y+h(z))_A$.
\end{itemize}

The graph $G_q(\rho)$ is $3$-partite with $q^{10}$ vertices in each part. Since $A$, $B$ and $C$ are disjoint, every listed edge has distinct endpoints. In each edge rule, the first coordinates of the endpoints determine $z$. Hence every edge has a unique representation. Thus $G_q(\rho)$ is simple.

Since $\F_{q^5}$ has odd characteristic, the scalars $2$ and $-1$ are nonzero. By Lemmas~\ref{lem:planar} and~\ref{lem:q-polynomial-perturbation}, applied with the $q$-polynomials $-L_\rho,L_\rho,0$, respectively, the polynomials $f$, $g$ and $h$ are planar.

By Lemma~\ref{lem:planar-graph}, applied to $f$, $g$ and $h$, the bipartite graph induced by each pair of parts is $C_4$-free.

It remains to count the triangles. Fix a vertex $(x,y)_A$. Let $r,s\in\F_{q^5}$ be the field elements used in the edge rules from $A$ to $B$ and from $B$ to $C$, respectively. These two edges lead first to $(x+r,y+f(r))_B$ and then to $(x+r+s,y+f(r)+g(s))_C$. If $t\in\F_{q^5}$ is the field element used in the edge rule from $C$ to $A$, the third edge ends at $(x+r+s+t,y+f(r)+g(s)+h(t))_A$. This endpoint equals $(x,y)_A$ if and only if $t=-r-s$ and $f(r)+g(s)+h(t)=0$. Thus the three edges form a triangle if and only if $f(r)+g(s)+h(-r-s)=0$.

Since $q^2+1$ is even, $P(-r-s)=P(r+s)$. By Observation~\ref{obs:q-polynomial-properties}, $L_\rho(r-s)=L_\rho(r)-L_\rho(s)$. By Lemma~\ref{lem:planar} and the definitions of $f$, $g$ and $h$, we obtain
\[
\begin{aligned}
f(r)+g(s)+h(-r-s)
&=2P(r)+2P(s)-P(r+s)-L_\rho(r)+L_\rho(s)\\
&=P(r-s)-L_\rho(r)+L_\rho(s)\\
&=P(r-s)-L_\rho(r-s).
\end{aligned}
\]
Thus the three edges form a triangle if and only if $P(r-s)=L_\rho(r-s)$.
For a fixed vertex $(x,y)_A$, the first-coordinate differences along the $A$--$B$ and $B$--$C$ edges uniquely determine $r$ and $s$. Hence distinct pairs $(r,s)$ satisfying this condition determine distinct triangles containing $(x,y)_A$. Every such triangle arises from one of these pairs.

Let $d=r-s$. By Lemma~\ref{lem:algebraic}, there are $q^2+1$ possible values of $d$. For each such value, $s$ is arbitrary and $r=d+s$ is uniquely determined, so there are $q^5(q^2+1)$ corresponding pairs $(r,s)$. Thus every vertex in $A$ lies in $q^5(q^2+1)$ triangles. Since every triangle contains exactly one vertex of $A$ and $|A|=q^{10}$, the total number of triangles is $q^{10}\cdot q^5(q^2+1)=q^{17}+q^{15}$. Therefore, $\triangle(q^{10})\ge q^{17}+q^{15}$, as required.

\paragraph{Declaration of generative AI use.} We used ChatGPT to assist in finding the auxiliary $q$-polynomial $L_\rho$, choosing the planar polynomials $f,g,h$ for the construction and proofreading the manuscript. All AI-assisted suggestions were independently reviewed and revised as necessary by the authors. The authors take full responsibility for the content of the paper.

\end{document}